\documentclass[10pt]{article}

\usepackage{amsthm,amsmath,amssymb}
\usepackage{yfonts}

\usepackage{graphicx,url}

\usepackage{tikz-cd}
\usepackage[colorlinks=true,citecolor=blue,linkcolor=blue,urlcolor=blue]{hyperref}

\newcommand{\arxiv}[1]{\href{http://arxiv.org/abs/#1}{\texttt{arXiv:#1}}}

\theoremstyle{plain}
\newtheorem{theorem}{Theorem}[section]
\newtheorem{lemma}[theorem]{Lemma}
\newtheorem{corollary}[theorem]{Corollary}

\theoremstyle{definition}

\newtheorem{example}[theorem]{Example}
\newtheorem{conjecture}[theorem]{Conjecture}

\theoremstyle{remark}
\newtheorem{remark}[theorem]{Remark}

\title{\bf The connection between the chromatic function and the Redei-Berge function}

\author{Stefan Mitrovi\'c \footnote{Corresponding author\\ Address : stefan.mitrovic@matf.bg.ac.rs}\\
\small Faculty of Mathematics\\[-0.8ex]
\small University of Belgrade\\[-0.8ex]
\small Serbia\\
\small\tt stefan.mitrovic@matf.bg.ac.rs\\
\and
Tanja Stojadinovi\'c\\
\small Faculty of Mathematics\\[-0.8ex]
\small University of Belgrade\\[-0.8ex]
\small Serbia\\
\small\tt tanja.stojadinovic@matf.bg.ac.rs}

\begin{document}

\maketitle

\begin{abstract}
There is a natural way to assign both graph and digraph to every poset. Furthermore, any graph has its chromatic function, while any digraph has its Redei-Berge function. On the level of posets, these two functions are almost identical. Here, we prove that this connection is actually a reflection of the connection between the noncommutative generalizations of these two functions. The simplicity of this relationship enables us to easily translate the properties proved for one of them to the case of the other. We perform such conversions regarding distinguishability, decomposition techniques and positivity questions. Among others, we obtain the converse of Redei's theorem, generalization of the triple deletion property and expressions for these functions in some special cases. 


\end{abstract}

\section{Introduction}

In 1995, Stanley defined the chromatic symmetric function of a simple graph as a generalization of the chromatic polynomial  \cite{SR}. In the last 30 years, many aspects of this function have been studied. The question that naturally emerges whenever a new combinatorial invariant is introduced is the question of distinguishing non-isomorphic objects. There are many papers dealing with this subject, for example  \cite{OS}, which also contains many decomposition techniques for the chromatic function. Furthermore, the abundance of new bases for the space of symmetric functions  that consist of the chromatic functions is given in \cite{CvW}. Finally, one of the most famous problems regarding this function is the question of positivity of its coefficients in many natural bases of $Sym$. Some of the papers that treat this topic are \cite{G}, \cite{MGP}, \cite{H}, \cite{ST}. Moreover, certain generalizations of the chromatic symmetric function have arisen in order to resolve this question \cite{SG}, \cite{SW}. Throughout this paper, we will be mostly interested in the one from \cite{SG}, which is defined in the space of symmetric functions in noncommuting variables.

Around the time of birth of the chromatic function, Stanley's student Chow defined the path-cycle symmetric function of a digraph \cite{TC}. The Redei-Berge symmetric function of a digraph is a recently introduced simplification of this function \cite{S}. Its expansion in the power sum basis easily yields two beautiful theorems from graph theory - Redei's theorem and Berge's theorem, hence the name, see \cite{LR}, \cite{CB}. The properties of the Redei-Berge function analogous to the previously described properties of the chromatic function have been studied in \cite{GS} and \cite{MS}. Finally, the generalization of the Redei-Berge function in noncommuting variables is introduced in \cite{M}.

There is a natural way to assign both graph and digraph to any poset. It turns out that, on the level of posets, the chromatic function and the Redei-Berge function are closely related. In this paper, we explore this connection in both commutative and noncommutative case. The results easily obtained for one of them yield unexpected consequences for the other. Although the ambient of graphs is more common, it appears that for some calculations the ambient of digraphs is more convenient. 

The paper is structured as follows. In Section 2, we recall the background of symmetric and quasisymmetric functions, graphs, digraphs and posets. In Section 3, we give another proof of the connection between the chromatic function and the Redei-Berge function. This proof uses descent sets of listings that appear in Stanley and Grinberg's definition of the Redei-Berge function. We originally obtained this result by discovering the antipodal connection between two combinatorial Hopf algebras. However, for the sake of simplicity, we decided to avoid introducing the theory of combinatorial Hopf algebras. Section 4 contains some examples of exploring the relationship between these two functions. In this way, we managed to prove a converse of the Redei's theorem. In addition, we investigate which properties are detectable from the chromatic and the Redei-Berge function and exploit the methods of constructing new bases for $Sym$ consisting of these functions. Finally, by comparing their expansions in the power sum basis, we obtain a combinatorial connection between the incomparability graph of $P$ and the digraph assigned to $P$. In Section 5, we give the analogous connection for the noncommutative versions of the chromatic and the Redei-Berge function. This yields yet another proof of the relationship of the ordinary versions of these functions. However, the techniques used for operating with noncommutative symmetric functions are completely different and often simpler since they allow the use of mathematical induction. The positivity of these two functions is the subject of Section 6. We convert the famous Stanley-Stembridge conjecture to the question of $h-$positivity of the Redei-Berge function, which could hopefully pave the new way for proving this hypothesis. The deletion-contraction property that is satisfied by the noncommuting Redei-Berge function turns out as a powerful tool for proving $h-$positivity for the ordinary Redei-Berge function of some classes of digraphs. In the last section, we exploit a decomposition of the Redei-Berge function into the sum of the Redei-Berge function of some very simple digraphs, called bags of sticks. In this manner, we obtain a generalization of the well-known triple deletion property of the chromatic function. Moreover, Chung and Graham's cover polynomial and Chow's path-cycle symmetric function enable us to explicitly calculate the Redei-Berge function and the Redei-Berge polynomial for bags of sticks.

\section{Preliminaries}

In this section, we review some basic notions and facts. A {\it set partition} $\pi=\{V_1,\ldots,V_k\}\vdash V$ of length $k=l(\pi)$
of a finite set $V$ is a set of disjoint nonempty subsets with $V_1\cup\ldots\cup V_k=V$. We write $\pi=V_1/\ldots/V_k$ and say that $V_i$'s are the blocks of $\pi$. A {\it set composition} $(V_1,\ldots,V_k)\models V$ is an ordered set partition. 

A \textit{composition} $\alpha\models n$ of length $k=l(\alpha)$, where $n\in\mathbb{N}$, is a sequence $\alpha=(\alpha_1,\ldots,\alpha_k)$ of
positive integers with $\alpha_1+\cdots+\alpha_k=n$. We say that $n$ is the weight of $\alpha$ and write $n=|\alpha|$. 
A \textit{partition} $\lambda\vdash n$ is a
composition $(\lambda_1,\ldots,\lambda_k)\models n$ such that $\lambda_1\geq \lambda_2\geq\cdots\geq
\lambda_k$. There is a bijection between the set of compositions of $n$, denoted by $\mathrm{Comp}(n)$ and the power set $2^{[n-1]}$ given by
$\mathrm{set}: (\alpha_1,\ldots,\alpha_k)\mapsto\{\alpha_1,\alpha_1+\alpha_2,\ldots, \alpha_1+\cdots+\alpha_{k-1}\}$. The inverse of this bijection is denoted by $I\mapsto\mathrm{comp}(I)$. We say that composition $\alpha\models n$ is \textit{finer} than composition  $\beta\models n$, and write $\alpha\leq\beta$, if $\mathrm{set}(\beta)\subseteq\mathrm{set}(\alpha)$. Equivalently, parts of $\beta$ are obtained by summing some adjacent parts of $\alpha$.

If $V$ is a set, a $V-$\textit{listing} is a list of all elements of $V$ with no repetitions, i.e. a bijective map $\sigma:[n]\rightarrow V$. We write $\Sigma_V$ for the set of all $V$-listings. A \textit{reversion} of a $V$-listing $\sigma=(\sigma_1,\ldots,\sigma_n)\in\Sigma_V$ is a $V$-listing $\sigma^\mathrm{rev}=(\sigma_n,\ldots,\sigma_1)$. For a subset $I\subseteq[n-1]$ we denote by $I^{\mathrm{op}}=\{n-i|i\in I\}$ its \textit{opposite set}. If $\alpha\models n$, we write $\alpha^{\mathrm{op}}$ for a composition whose corresponding set is $\mathrm{set}(\alpha)^{\mathrm{op}}$. In other words, if $\alpha=(\alpha_1, \ldots, \alpha_k)$, then $\alpha^{\mathrm{op}}=(\alpha_k, \ldots, \alpha_1)$.

The ambient spaces of the initial part of our paper are the spaces of symmetric and quasisymmetric functions, $Sym$ and $QSym$. For basics of symmetric and quasisymmetric functions see \cite{EC}. A composition $\alpha=(\alpha_1,\ldots,\alpha_k)\models n$ defines the \textit{monomial quasisymmetric function}

\[M_\alpha=\sum_{i_1<\cdots<
i_k}x_{i_1}^{\alpha_1}\cdots x_{i_k}^{\alpha_k}.\] Alternatively, we write $M_I=M_{\mathrm{comp}(I)},\ I\subseteq[n-1]$. Another basis of $QSym$ is made up of \textit{fundamental quasisymmetric functions }

\begin{equation}\label{fundamental}
F_I=\sum_{\substack{1\leq i_1\leq i_2\leq\cdots\leq i_n\\
                  i_j<i_{j+1} \ \mathrm{for \ each} \ j\in I}}x_{i_1}x_{i_2}\cdots x_{i_n}, \quad I\subseteq[n-1],
\end{equation}
which are expressed in the monomial basis by
\begin{equation}\label{montofund}
F_I=\sum_{I\subseteq J}M_J, \quad I\subseteq[n-1].
\end{equation}

On the space $QSym$, there is an automorphism $\omega: QSym\rightarrow QSym$, defined on this basis by $\omega(F_I)=F_{I^c},$ see \cite{SR}. Clearly, $\omega$ is an involution.

The \textit{principal specialization} $\mathrm{ps}^1:QSym\rightarrow\mathbf{k}[m]$ is an algebra homomorphism to the polynomial algebra defined by
\[\mathrm{ps}^1(\Phi)(m)=\Phi(\underbrace{1,\ldots,1}_{\text{$m$ ones}},0,0,\ldots).\]

The algebra of symmetric functions $Sym$ is a subalgebra of $QSym$  that consists of quasisymmetric functions which are invariant under the action of permutations on the set of variables. There are many bases of this vector space. For partition $\lambda$, \textit{monomial symmetric function} $m_\lambda$ is defined as \[m_\lambda=\sum_{\alpha\sim\lambda}M_\alpha,\] where the summation runs over all compositions $\alpha$ whose parts, when arranged in nonincreasing order, yield $\lambda$. The $i$th \textit{elementary symmetric function} is given by 
\[e_0=1\hspace{5mm} \textrm{ and } \hspace{5mm}e_i=\sum_{j_1<j_2<\cdots<j_i}x_{j_1}x_{j_2}\cdots x_{j_i} \] for $i\geq 1$.
For partition $\lambda=(\lambda_1, \lambda_2, \ldots, \lambda_k)$, we define
\[e_{\lambda}=e_{\lambda_1}e_{\lambda_2}\cdots e_{\lambda_k}.\]
 The $i$th \textit{power sum symmetric function} is 
\[p_0=1 \hspace{5mm}\text{ and }\hspace{5mm}p_i=\sum_{j=1}^{\infty}x_j^i\]
for $i\geq 1$ and again, for $\lambda=(\lambda_1, \lambda_2, \ldots, \lambda_k)$ we have
\[p_{\lambda}=p_{\lambda_1}p_{\lambda_2}\cdots p_{\lambda_k}.\] 
 The \textit{complete homogeneous symmetric function} $h_\lambda$ is  \[h_\lambda=\omega(e_\lambda).\]
Finally, for $\lambda=(\lambda_1, \ldots, \lambda_k)$, we define \textit{Schur function} $s_\lambda$ by \[s_\lambda=\mathrm{det}(h_{\lambda_i-i+j})_{i, j=1}^k.\]  It turns out that $\omega(p_\lambda)=(-1)^{|\lambda|-l(\lambda)}p_\lambda$ and $\omega(s_\lambda)=s_{\lambda'}$, where $\lambda'$ denotes the conjugate partition of $\lambda$ whose parts $\lambda_i'$ are $|\{j\mid \lambda_j\geq i\}|$.

 The collections $\{m_\lambda\}, \{e_\lambda\}, \{p_\lambda\}, \{h_\lambda\}$ and $\{s_\lambda\}$, when $\lambda$ runs over all partitions are bases of $Sym$. If $u=\{u_i\}$ is a basis of some vector space $V$ and $v\in V$, we write $[u_i]v$ for the coefficient of $u_i$ in the $u-$expansion of $v$.

A \textit{graph} $G$ is a pair $G=(V, E)$, where $V$ is a finite set and $E$ is a collection of unordered pairs of elements of $V$. Elements of $V$ are called \textit{vertices} and elements of $E$ are called \textit{edges}. Two graphs $G=(V, E)$ and $H=(V', E')$ are \textit{isomorphic} if there is a bijection $f: V\rightarrow V'$ such that $\{u, v\}\in E\iff \{f(u), f(v)\}\in E'$.
The \textit{restriction} of a graph $G=(V,E)$ on a subset $S\subseteq V$ is the graph $G|_S=(S,E|_S)$, where $E|_S=\{\{u,v\}\in E\  |\ u,v\in S\}$.
We say that $G=(V, E)$ is \textit{discrete} if $E=\emptyset$ and that it is \textit{complete} (or a \textit{clique}) if $E=\binom{V}{2}$. 

If $G=(V, E)$, we call a function $f: V\rightarrow\mathbb{N}$ a \textit{proper coloring} of $G$ if $\{u, v\}\in E\implies f(u)\neq f(v)$. The \textit{chromatic symmetric function} $X_G$ of graph $G=(V, E)$ with vertex set $V=\{v_1, \ldots, v_n\}$ is \[X_G=\sum_f x_{f(v_1)}x_{f(v_2)}\cdots x_{f(v_n)},\]
where the sum is over all proper colorings $f$ of $G$ \cite{SR}. The \textit{chromatic polynomial} of a graph $G$ is the principal specialization of its chromatic function \[\chi_G(m)=\mathrm{ps}^1(X_G)(m),\] i.e. the number of ways to color a graph properly with at most $m$ different colors.

A {\it digraph} $X$ is a pair $X=(V,E)$, where $V$ is a finite set and $E$ is a collection
$E\subseteq V\times V$. Elements $u\in V$ are \textit{vertices} and
elements $(u,v)\in E$ are \textit{edges} of the digraph $X$. We say that an edge $(u, u)$ is a \textit{loop}. Two digraphs $X=(V, E)$ and $Y=(V', E')$ are \textit{isomorphic} if there is a bijection $f: V\rightarrow V'$ such that $(u, v)\in E\iff (f(u), f(v))\in E'$. 

If $X=(V, E)$ is a digraph, its \textit{complementary digraph} is the digraph $\overline{X}=(V, (V\times V)\setminus E)$ and its \textit{opposite digraph} is $X^{op}=(V, E')$, where $E'=\{(v, u) \ \mid \ (u, v)\in E\}$. The \textit{restriction} of a digraph $X=(V,E)$ on a subset $S\subseteq V$ is the
digraph $X|_S=(S,E|_S)$, where $E|_S=\{(u,v)\in E\  |\ u,v\in S\}$. For digraphs $X=(V,E)$ and $Y=(V',E')$, the product $X\cdot Y$ is defined as the digraph on the disjoint union $V\sqcup V'$ with the set of directed edges \[E\cup E'\cup\{(u,v)\ |\ u\in V, v\in V'\}.\] 

If $X=(V, E)$, for a $V$-listing $\sigma=(\sigma_1,\ldots,\sigma_n)\in\Sigma_V$, define the $X$-{\it descent set} as

\[X\mathrm{Des}(\sigma)=\{1\leq i\leq n-1\ | \ (\sigma_i,\sigma_{i+1})\in E\}.\]

Grinberg and Stanley assigned to a digraph $X$ a generating function for $X$-descent sets  expanded in the basis of fundamental quasisymmetric functions

\begin{equation}\label{descents}
U_X=\sum_{\sigma\in\Sigma_V}F_{X\mathrm{Des}(\sigma)}
\end{equation}
and named it the \textit{Redei-Berge symmetric function}, see \cite{S}. This definition, together with the fact that $\omega(F_I)=F_{I^c}$, easily gives the following equation

\begin{equation}\label{antipod}
\omega(U_X)=U_{\overline{X}}.
\end{equation}

The \textit{Redei-Berge polynomial} of a digraph $X$ is the principal specialization of its Redei-Berge function: \[u_X(m)=\mathrm{ps}^1(U_X)(m).\]

\begin{theorem}\label{antipodpolinom}
    \cite[Theorem 5.8]{GS}For any digraph $X=(V, E)$, \[u_X(m)=(-1)^{|V|}u_{\overline{X}}(-m).\]
\end{theorem}

 Let $X=(V, E)$ be a digraph and let $\textfrak{S}_V$ be the group of permutations of $V$. Then, we define
    \[\textfrak{S}_V(X)=\{\pi\in \textfrak{S}_V\mid \textrm{ each non-trivial cycle of }\pi \textrm{ is a cycle of } X\},\]
    \[\textfrak{S}_V(X, \overline{X})=\{\pi\in \textfrak{S}_V\mid \textrm{ each cycle of }\pi \textrm{ is a cycle of } X \textrm{ or a cycle of }\overline{X}\}.\]

The expansion of the Redei-Berge function in the power sum basis is given in \cite{S}. Let $\mathrm{type}(\pi)$ denote the partition whose entries are the lengths of the
cycles of $\pi$. 

\begin{theorem}\cite[Theorem 1.31]{S}\label{pbaza} Let $X=(V, E)$ be a digraph. For any $\pi\in \textfrak{S}_V$, let $\varphi(\pi):=\sum_{\gamma}(\ell (\gamma)-1),$ where the summation runs over all cycles $\gamma$ of $\pi$ that are cycles in $X$ and $\ell(\gamma)$ denotes the length of the cycle $\gamma$. Then, \[U_X=\sum_{\pi\in\textfrak{S}_V(X, \overline{X})}(-1)^{\varphi(\pi)}p_{\mathrm{type}(\pi)}.\]
    
\end{theorem}

Let $P$ be a finite poset.  The relation that induces poset $P$ will be denoted as $\leq_P$ and the corresponding strict order relation as $<_P$. Let $\mathrm{Min}(P)$ denote the set of all minimal elements of $P$ and $\mathrm{Max}(P)$ the set of all maximal elements of $P$.
For a $P$-listing $\sigma=(\sigma_1, \ldots, \sigma_n)$, we say that it is a \textit{linear extension} of $P$ if there do not exist $i<j$ such that $\sigma_{j}\leq_P\sigma_i$. A \textit{quasi-linear extension} of $P$ is defined in \cite{MS} as a $P$-listing $(\sigma_1, \ldots, \sigma_n)$  such that $\sigma_{i+1}\nleq_P \sigma_i$ for all $i\in[n-1]$. Clearly, every linear extension is also a quasi-linear extension. We denote by $P^*$ the \textit{dual poset} of $P$, where $p\leq_{P^*}q$ if and only if $q\leq_P p$. For $p, q\in P$, $[p, q]=\{r\in P\mid p\leq_Pr\leq_Pq\}.$
The \textit{incomparability graph} of $P$, denoted by $\mathrm{inc}(P)$, is a graph with $P$ as the set of vertices, while $\{i, j\}$ is an edge of $\mathrm{inc}(P)$ if and only if $i$ and $j$ are incomparable in $P$.

The \textit{ordinal sum} $PQ$ of two posets $P$ on $[m]$ and $Q$ on $[n]$ is the poset on $[m+n]$ with $\leq_{PQ}$ defined as follows: 
\begin{itemize}
\item  for $x, y\in[m]$, $x\leq_{PQ}y$ if and only if $x\leq_P y$,
\item  for $x, y \in [n]$,  $x+m\leq_{PQ} y+m$ if and only if $x\leq_Q y$  and
\item for all $x\in [m]$ and $y\in \{m+1, \ldots, m+n\}$, $x\leq_{PQ} y$ 
\end{itemize}
 A poset $P$ is \textit{irreducible} if it is not an ordinal sum of two nonempty posets.

The authors in \cite{MS} defined the Redei-Berge function $U_P$ of a poset $P$ as the Redei-Berge function $U_{D_P}$ of an appropriate digraph $D_P$ assigned to $P$. The set of vertices of $D_P$ is $P$, while $(i, j)$ is an edge of $D_P$ if and only if $i<_P j$. Note that \[U_P=\sum_{\sigma\in\Sigma_P}F_{P\mathrm{Des}(\sigma)},\]
where $P\mathrm{Des}(\sigma)=\{1\leq i\leq n-1\ | \ \sigma_i <_P\sigma_{i+1}\}.$


\section{The connection between $U_P$ and $X_{\mathrm{inc}(P)}$}

This section is a starting point for our further research. We point out that a different approach to the proofs of the theorems in this section can be found in \cite{TC}. 

Recall that for a given graph $G=(V, E)$, $S\subseteq V$ is called \textit{stable} if $G|_S$ is a discrete graph. A partition, or a composition of $V$ is  \textit{stable} if its every block is stable. If $P$ is a poset and $\sigma\in \Sigma_P$, let $A_P(\sigma)=\{i\in[n-1]\mid \sigma_i \nleq_P \sigma_{i+1}\}$.

\begin{lemma} \label{fbaza}
    For any poset $P$, \[X_{\mathrm{inc}(P)}=\sum_{\sigma\in \Sigma_P}F_{A_P(\sigma)}.\] 
\end{lemma}

\begin{proof}
    If $\alpha=(\alpha_1, \ldots, \alpha_k)$ is a composition of $|P|$, then \[[M_\alpha](\sum_{\sigma\in \Sigma_P}F_{A_P(\sigma)})=\#\{\sigma\in\Sigma_P\mid A_P(\sigma)\subseteq \mathrm{set}(\alpha)\}.\] Therefore, we need to count the listings $\sigma$ such that $A_P(\sigma)\subseteq\{\alpha_1, \alpha_1+\alpha_2, \ldots, \alpha_1+\cdots+\alpha_{k-1}\}$. This condition is equivalent to 

\begin{center}
  $\sigma_1\leq_P\sigma_2\leq_P\cdots\leq_P\sigma_{\alpha_1}$ \\  $\sigma_{\alpha_1+1}\leq_P\sigma_{\alpha_1+2}\leq_P\cdots\leq_P\sigma_{\alpha_1+\alpha_2}$ \\
   \;\;\vdots \notag \\
  $\sigma_{\alpha_1+\cdots+\alpha_{k-1}+1}\leq_P\sigma_{\alpha_1+\cdots+\alpha_{k-1}+2}\leq_P\ldots\leq_P\sigma_{n}.$
\end{center}
Hence, there are exactly $\#\{(P_1, \ldots, P_k)\models P\mid P_i \textrm{ is a chain with }\alpha_i\ \textrm{elements}\}$ such listings $\sigma$. On the other hand, the definition of the chromatic symmetric function immediately gives $[M_\alpha]X_{\mathrm{inc}(P)}=\#\{(P_1, \ldots, P_k)\models P\mid (P_1, \ldots, P_k) \textrm{ is a stable composition in $\mathrm{inc}(P)$ of type } \alpha\}$. The fact that stable subsets of $\mathrm{inc}(P)$ correspond to chains in $P$ completes the proof.
 \end{proof}

This lemma easily gives the following.

\begin{theorem}\label{incomparability}
    If $P$ is a poset  then  \[X_{\mathrm{inc}(P)}=\omega(U_P).\]
\end{theorem}

\begin{proof}
    Since $A_P(\sigma)=\overline{D_P}\mathrm{Des(\sigma)}$, we have \[X_{\mathrm{inc}(P)}=\sum_{\sigma\in \Sigma_P}F_{A_P(\sigma)}=\sum_{\sigma\in\Sigma_P}F_{\overline{D_P}\mathrm{Des}(\sigma)}=U_{\overline{D_P}}=\]\[=\omega(U_{D_P})=\omega(U_P),\] where the fourth equality follows from Equation \ref{antipod}.
\end{proof}

\begin{corollary} \label{polinomi}
    If $P$ is a poset, then \[\chi_{\mathrm{inc}(P)}(m)=(-1)^{|P|}u_P(-m).\]
\end{corollary}

\begin{proof}
    Theorem \ref{incomparability} and Equation \ref{antipod} yield $X_{\mathrm{inc}(P)}=\omega(U_P)=\omega(U_{D_P})=U_{\overline{D_P}}.$
    Therefore, $\chi_{\mathrm{inc}(P)}(m)=u_{\overline{D_P}}(m).$ However, Theorem \ref{antipodpolinom} implies that $u_{\overline{D_P}}(m)=(-1)^{|P|}u_{D_P}(-m),$ which is exactly $(-1)^{|P|}u_P(-m)$.
\end{proof}

\begin{remark}
    For readers familiar with the theory of (combinatorial) Hopf algebras, we note that Theorem \ref{incomparability} can be seen as the following commutative diagram of Hopf algebras. Here, $\mathcal{P}_1$ and $\mathcal{P}_2$ denote combinatorial Hopf algebras of posets from \cite{ABM} and \cite{MS} respectively, while $\mathcal{D}$, $\mathcal{G}$ and $Sym$ are combinatorial Hopf algebras of digraphs \cite{GS}, graphs \cite{ABS} and symmetric functions \cite{ABS}.
\begin{center}
\begin{tikzcd}
             & \mathcal{P}_1 \arrow[rr, "Id"] \arrow[ld] &  & \mathcal{P}_2 \arrow[ll] \arrow[rd] &              \\
\mathcal{G} \arrow[rd] &                              &  &                         & \mathcal{D} \arrow[ld] \\
             & Sym \arrow[rr, "\omega"]     &  & Sym \arrow[ll]          &             
\end{tikzcd}
\end{center}
\end{remark}
In this diagram, $\mathcal{P}_1\rightarrow\mathcal{G}$ stands for $P\mapsto \mathrm{inc}(P)$, $\mathcal{P}_2\rightarrow\mathcal{D}$ for $P\mapsto D_P$, while $\mathcal{G}\rightarrow Sym$ and $\mathcal{D}\rightarrow Sym$ are respectively $G\mapsto X_G$ and $D\mapsto U_D$. Moreover, these four morphisms of Hopf algebras are even morphisms of combinatorial Hopf algebras. The remaining two morphisms do not satisfy this property since $\omega$ is not compatibile with the character on $Sym$, while the characters on $\mathcal{P}_1$ and $\mathcal{P}_2$ are different.

\section{Going back and forth}

According to Theorem \ref{incomparability}, the theory of the chromatic function can be used in examining properties of the Redei-Berge function and vice versa. The tools developed in the environment of one of these functions, when translated to the language of the other, yield unexpected results. In this section, we give some representative examples of exploiting this interconnection.

\begin{theorem}
    If $P$ is a poset that is not a chain, then the number of quasi-linear extensions of $P$ is even.
\end{theorem}

\begin{proof}
    Since the chromatic polynomial has integer coefficients, $\chi_{\mathrm{inc}(P)}(1)$ and $\chi_{\mathrm{inc}(P)}(-1)=(-1)^{|P|}u_P(1)$ are of the same parity. However, $u_P(1)$ is the number of quasi-linear extensions of $P$ \cite{MS}, while $\chi_{\mathrm{inc}(P)}(1)$ is the number of proper colorings of $\mathrm{inc}(P)$ with one color, i.e. 1 if $P$ is a chain and 0 otherwise.
\end{proof}
Recall that a \textit{tournament} is a loopless digraph $(V, E)$ such that for every two distinct $u, v\in V$ exactly one of $(u, v)$ and $(v, u)$ is in $E$. The previous theorem can be seen as a converse of Redei's theorem, which states that the number of Hamiltonian paths in any tournament is odd. Namely, the only posets whose digraphs are tournaments are exactly the chains, and there is only one quasi-linear extension for them.

If two posets have the same Redei-Berge function, then they have the same number of chains of any length \cite[Corollary 4.9]{MS}. Consequently, if two posets that are disjoint union of chains have the same Redei-Berge function, then they are isomorphic \cite[Theorem 4.11]{MS}. Therefore, this property also holds for the chromatic function of their incomparability graphs. Recall that a graph $G=(V, E)$ is \textit{complete} $k-$\textit{partite} if there exists a partition $(V_1, \ldots, V_k)\vdash V$ for which $E=\{\{u, v\}\mid u\in V_i, v\in V_j, i\neq j\}$.

\begin{theorem}
    Let $G$ and $H$ be complete $k-$ and $l-$partite graphs respectively. If $X_G=X_H$, then $G$ and $H$ are isomorphic.
\end{theorem}

\begin{proof} 
    This follows immediately from previous consideration, since $G$ is the incomparability graph of a poset that consists of $k$ disjoint chains, while $H$ is the incomparability graph of a poset consisting of $l$ disjoint chains.
\end{proof}

In this spirit, we also note that two graphs with equal chromatic function have the same number of triangles as well \cite[Corollary 2.3]{OS}. Consequently,

\begin{theorem}
    If two posets have the same Redei-Berge function, then they have the same number of unordered triples in which the elements are incomparable.
\end{theorem}

We now focus our attention on various bases of $Sym$ whose elements are the chromatic functions and the Redei-Berge functions.

\begin{lemma} \label{povezan}
    A poset $P$ has a connected incomparability graph if and only if it is irreducible.
\end{lemma}

\begin{proof}
    If $P=P_1P_2$, then clearly $\mathrm{inc}(P)=\mathrm{inc}(P_1)\cup\mathrm{inc}(P_2)$ is disconnected. Conversely, let $\mathrm{inc}(P)$ be disconnected, $O_P(p)=\{q\in P\mid p\leq_Pq\}$ and let $\sigma\in\Sigma_P$ be such that $|O_P(\sigma_i)|$ is non-increasing. Note that such $\sigma$ is a linear extension of $P$ since if $p<_Pq$, then $O_P(q)\subsetneq O_P(p)$. 

    We first prove that there exists $i\in [n-1]$ such that for all $j>i$, $\sigma_i\leq_P\sigma_j$. If not, then for all $i\in [n-1]$, there exists $j>i$ such that $\sigma_i\nleq_P\sigma_j$. Since $\sigma$ is a linear extension, $\sigma_j\nleq_P\sigma_i$. Therefore, $\{\sigma_i, \sigma_j\}$ is an edge in $\mathrm{inc}(P)$. If $j\neq n$, in the same way we can find $k>j$ such that $\{\sigma_j, \sigma_k\}$ is an edge in $\mathrm{inc}(P)$ and so on. Thus, for every $i\in [n-1]$ we obtain a path from $\sigma_i$ to $\sigma_n$ in $\mathrm{inc}(P)$, which contradicts the fact that $\mathrm{inc}(P)$ is disconnected. Consequently, there exists $i\in [n-1]$ such that for all $j>i$, $\sigma_i\leq_P\sigma_j$.
    
    Furthermore, if $\sigma_{i-1}\leq_P\sigma_i$, transitivity yields $\sigma_{i-1}\leq_P\sigma_j$ for all $j>i$. If $\sigma_{i-1}\nleq_P\sigma_i$, then still $\sigma_{i-i}\leq_P\sigma_j$ for all $j>i$ since $|O_P(\sigma_{i-1})|\geq|O_P(\sigma_i)|=n-i$. Similarly, if $\sigma_{i-2}\leq_P\sigma_{i-1}$ or $\sigma_{i-2}\leq_P\sigma_i$, transitivity implies $\sigma_{i-2}\leq_P\sigma_j$ for all $j>i$. If $\sigma_{i-2}\nleq_P\sigma_{i-1}$ and $\sigma_{i-2}\nleq_P\sigma_i$, then $\sigma_{i-2}\leq_P\sigma_j$ for all $j>i$ since $|O_P(\sigma_{i-2})|\geq|O_P(\sigma_{i-1})|$. Analogously, we obtain that $\sigma_k\leq_P\sigma_j$ for all $k\leq i$ and for all $j>i$, and therefore $P=P|_{\{\sigma_1, \ldots, \sigma_{i}\}} \cdot P|_{\{\sigma_{i+1}, \ldots, \sigma_n\}}$.
\end{proof}

The authors in \cite[Theorem 5]{CvW} proved that if $(G_n)_{n\in\mathbb{N}}$ is a list of connected graphs such that $G_n$ has $n$ vertices, then $(X_{G_n})_{n\in\mathbb{N}}$ is an algebraic basis of $Sym$. Previous lemma then yields the following.

\begin{corollary}\label{stefani}
    If $(P_n)_{n\in\mathbb{N}}$ is any list of irreducible posets, then $(U_{P_n})_{n\in\mathbb{N}}$ is an algebraic basis of $Sym$.
\end{corollary}

Further, in \cite[Theorem 4.16]{MS} is proved that the list $(U_{X_n})_{n\in\mathbb{N}}$ is an algebraic basis of $Sym$ given that $X_n$ is a digraph with $n$ vertices such that the number of Hamiltonian cycles of $X_n$ and the number of Hamiltonian cycles of $\overline{X_n}$ are different. If $P$ is a poset, then $D_P$ does not contain any non trivial cycle. Therefore, if $(P_n)_{n\in\mathbb{N}}$ is a list of posets such that $P_n$ has $n$ elements and $\overline{D_{P_n}}$ has at least one Hamiltonian cycle for $n\geq 2$, then $(U_{P_n})_{n\in\mathbb{N}}$, and consequently, $(X_{\mathrm{inc}(P_n)})_{n\in\mathbb{N}}$, are two algebraic basis of $Sym$. This result and the one given in Corollary \ref{stefani} are essentially the same.

\begin{theorem} \label{hamired}
    If $P$ is a poset, then $\overline{D_P}$ contains at least one Hamiltonian cycle if and only if $P$ is irreducible.
\end{theorem}

\begin{proof}
    Clearly, if $P=P_1P_2$, then $\overline{D_P}$ does not contain Hamiltonian cycle. Conversely, if $P$ is irreducible, then $\mathrm{inc}(P)$ is connected by Lemma \ref{povezan}. If $G$ is a connected graph, then $[p_n]X_G\neq 0$ \cite[Lemma 2]{CvW}.
    According to Theorem \ref{pbaza}, $[p_n]U_P$ is the number of Hamiltonian cycles in $\overline{D_P}.$ Since $[p_n]U_P=\pm[p_n]X_{\mathrm{inc}(P)}$, it follows that the number of Hamiltonian cycles in $\overline{D_P}$ is non-zero.
\end{proof}

We can now go one step further and use Stanley's Broken Cycle theorem from \cite{SR}. For a given graph $G=(V, E)$, a \textit{labeling of} $G$ is a bijective map $\alpha: E\rightarrow \{1, \ldots, |E|\}$. A \textit{broken cycle} is a cycle in $G$ whose largest edge with respect to labeling $\alpha$ is removed. The \textit{broken cycle complex} $B_G$ of $G$ is the collection of all subsets of $E$ that do not contain a broken cycle.

\begin{theorem}  \cite[Theorem 2.9]{SR} \label{razbijen}  For any graph $G=(V, E)$,
   \[X_G=\sum_{S\in B_G}(-1)^{|S|}p_{\lambda(S)},\]where $\lambda(S)$ denotes the partition whose entries correspond to the sizes of the connected components of $(V, S)$.
\end{theorem}

The following theorem generalizes Theorem \ref{hamired} and gives another combinatorial connection between $\mathrm{inc}(P)$ and $D_P$.

\begin{theorem}
    For any poset $P$ and for any $\lambda\vdash |P|$, \[\#\{S\in B_{\mathrm{inc}(P)}\mid \lambda(S)=\lambda\}=\#\{\pi\in\mathfrak{S}_V(\overline{D_P})\mid \mathrm{type}(\pi)=\lambda\}.\]
\end{theorem}
\begin{proof}
Note that if $S\in B_G$, $(V, S)$ is a forest. Therefore, if $\lambda(S)=\lambda$, then $|S|=n-l(\lambda)$ and Theorem \ref{razbijen} implies \[[p_\lambda]X_G=(-1)^{n-l(\lambda)}\#\{S\in B_G\mid \lambda(S)=\lambda\}.\] If $P$ is a poset, $D_P$ does not contain cycles, hence Theorem \ref{pbaza} gives \[[p_\lambda] U_P=\#\{\pi\in\textfrak{S}_V(\overline{D_P})\mid \mathrm{type}(\pi)=\lambda\}.\] From $\omega(X_G)=U_P$ and $\omega(p_\lambda)=(-1)^{n-l(\lambda)}p_{\lambda}$, it follows $[p_\lambda]X_{\mathrm{inc}(P)}=(-1)^{n-l(\lambda)}[p_\lambda]U_P$, which completes the proof.

\end{proof}

For example, if $\lambda=(2, 1, \ldots, 1)$, it is not difficult to see that these numbers are actually $\#\{u, v\in P\mid u \text{ and } v \text{ are incomparable}\}$.  On the other hand, for $\lambda=(|P|)$, previous theorem implies that the number of Hamiltonian cycles of $\overline{D_P}$ is the same as the number of spanning subtrees of $\mathrm{inc}(P)$ that do not contain a broken cycle.

\section{Noncommutative analogues}

Although working with noncommuting variables may seem like an additional difficulty, it allows us to keep track of the position of each variable in every monomial. This bigger picture offers the possibility of proving some properties that the commutative versions do not possess. Furthermore, many properties of commutative functions can be easily obtained in this setting.

\subsection{Noncommutative symmetric functions}

We now introduce the space of our particular interest - the vector space of the noncommutative symmetric functions. For basics of symmetric functions in noncommuting variables see \cite{BZ}, \cite{SG}, \cite{RS}. Noncommutative symmetric functions are indexed by the elements of partition lattice. Let $\Pi_n$ denote the collection of all set partitions of $[n]$. This collection is ordered by refinement, which we denote as $\leq$. More precisely, for $\pi, \sigma\in \Pi_n$, we write $\pi\leq\sigma$ if every block of $\pi$ is contained in some block of $\sigma$. With respect to this ordering $\Pi_n$ becomes a lattice. For $\pi, \sigma\in \Pi_n$, denote by $\pi\land\sigma$ their meet (greatest lower bound) and by $\pi\lor\sigma$ their join (least upper bound).
 
  Let $\lambda(\pi)$ denote the integer partition of $n$ whose parts correspond to the block sizes of $\pi$. If $\lambda(\pi)=(1^{r_1}, 2^{r_2}, \ldots, n^{r_n})$, we will write $|\pi|$ for $r_1!r_2!\cdots r_n!$ and $\pi!$ for $1!^{r_1}2!^{r_2}\cdots n!^{r_n}$.

  The \textit{noncommutative monomial symmetric function}, $m_{\pi}$, is defined as \[m_{\pi}=\sum_{i_1, i_2, \ldots, i_n}x_{i_1}x_{i_2}\cdots x_{i_n},\]
where the sum is over all sequences $i_1, i_2, \ldots, i_n$ of positive integers such that $i_j=i_k$ if and only if $j$ and $k$ are in the same block of $\pi$. These functions are linearly independent over $\mathbb{C}$ and we call their span the \textit{algebra of noncommutative symmetric functions} $NCSym$. 

The second basis we will be interested in contains the \textit{noncommutative elementary symmetric functions}, defined by \[e_{\pi}=\sum_{i_{1}, i_{2}, \ldots, i_n}x_{i_1}x_{i_2}\cdots x_{i_n},\]
where the sum runs over all sequences $i_1, i_2, \ldots, i_n$ of positive integers such that $i_j\neq i_k$ if $j$ and $k$ are in the same block of $\pi$.

Another basis of this space consists of the \textit{noncommutative power sum symmetric functions} given by \begin{equation}
    \label{pprekom}p_{\pi}=\sum_{\pi\leq\sigma}m_\sigma=\sum_{i_1, i_2, \ldots, i_n}x_{i_1}x_{i_2}\cdots x_{i_n},\end{equation}
where the second sum is over all positive integer sequences such that $i_j=i_k$ whenever $j$ and $k$ are in the same block of $\pi$.

Our final basis contains the \textit{noncommutative complete homogeneous symmetric functions} which are given by \[h_{\pi}=\sum_{\sigma\in \Pi_n}(\sigma\land\pi)!m_{\sigma}.\]

Letting the variables commute transforms $m_\pi$ to $|\pi|m_{\lambda(\pi)}$, $e_\pi$ into $\pi !e_{\lambda(\pi)}$, $p_{\pi}$ into $p_{\lambda(\pi)}$ and $h_\pi$ to $\pi !h_{\lambda(\pi)}.$ 

These functions are invariant under the standard action of the symmetric group, and hence symmetric in the usual sense. On the other hand, we need to define another action of the symmetric group $S_n$ on this space, which permutes the positions of the variables. For $\delta\in S_n$, let $\delta\circ(x_{i_1}x_{i_2}\cdots x_{i_n})$ be the monomial $x_{i_{\delta^{-1}(1)}}x_{i_{\delta^{-1}(2)}}\cdots x_{i_{\delta^{-1}(n)}}$. This action is extended linearly to the whole space of symmetric functions in noncommuting variables.

In \cite[Theorem 3.4]{RS}, it is shown that \begin{equation} \label{pprekoh}
    p_{\pi}=\frac{1}{|\mu(\widehat{0}, \pi)|}\sum_{\sigma\leq\pi}\mu(\sigma, \pi)h_{\sigma}, \text{ and } h_{\pi}=\sum_{\sigma\leq\pi}|\mu(\widehat{0}, \sigma)|p_\sigma,
\end{equation}
where $\widehat{0}$ denotes the minimal partition of $[n]$ and $\mu$ denotes the Möbius function of $\Pi_n$.

We denote by $\rho : NCSym\rightarrow Sym$ the operator that allows the variables commute. On $NCSym$, the authors in \cite{RS} defined an automorphism $\omega$ by $\omega(e_\pi)=h_\pi$. It turns out that this automorphism is an involution that is compatible with $\omega$ on $Sym$. In other words, $\omega\circ\rho=\rho\circ\omega$, where the first $\omega$ is from $Sym$ and the second one is from $NCSym$. Moreover, they proved that if $\pi$ is a partition of $[n]$, then $\omega(p_\pi)=(-1)^{n-l(\pi)}p_{\pi}$ \cite[Theorem 3.5]{RS}.

Eventually, we define the induction operation on monomials, denoted by $\uparrow$, as \[(x_{i_1}x_{i_2}\cdots x_{i_{n-1}}x_{i_{n}})\uparrow=x_{i_1}x_{i_2}\cdots x_{i_{n-1}}x_{i_{n}}^2\] and extend it linearly to the space of noncommutative symmetric functions.  We can easily see how this operation affects the monomial and the power sum symmetric functions. For $\pi\in\Pi_{n}$, let $\pi+(n+1)\in\Pi_{n+1}$ denote the partition $\pi$ with $n+1$ inserted into the block containing $n$. Then, $m_{\pi}\uparrow=m_{\pi+(n+1)}$ and $p_{\pi}\uparrow=p_{\pi+(n+1)}$. In combination with Equation \ref{pprekoh}, this yields:

 \begin{equation}\label{indukcijah}
       h_{\pi}\uparrow=\sum_{\sigma\leq\pi}\frac{|\mu(\widehat{0}, \sigma)|}{|\mu(\widehat{0}, \sigma+(n+1))|}\sum_{\tau\leq \sigma+(n+1)}\mu(\tau, \sigma+(n+1))h_\tau.
 \end{equation}

 \begin{lemma}\label{indukcijaomega}
     If $f\in NCSym$, then $\omega(f\uparrow)=-\omega(f)\uparrow$.
 \end{lemma}
 \begin{proof}
     Let $f=\sum c_\pi p_\pi$. Since $l(\pi+(n+1))=l(\pi)$, \[\omega(f\uparrow)=\omega(\sum c_\pi p_{\pi+(n+1)})=\sum c_{\pi}(-1)^{n+1-l(\pi)}p_{\pi+(n+1)}.\] On the other hand, \[\omega(f)\uparrow=\omega(\sum c_\pi p_\pi)\uparrow=\sum c_\pi (-1)^{n-l(\pi)}p_{\pi+(n+1)},\] which completes the proof.
 \end{proof}

 \subsection{The chromatic function and the Redei-Berge function in noncommuting variables}

 Gebhard and Sagan introduced the noncommutative version $Y_G$ of the chromatic function of a graph in \cite{SG}.

 Let $G$ be a graph with vertices labeled $v_1, v_2, \ldots, v_n$ in a fixed order. The \textit{chromatic symmetric function} of $G$ in noncommuting variables is
    \[Y_G=\sum_k x_{k(v_1)}x_{k(v_2)}\cdots x_{k(v_n)},\]
    where the sum is over all proper colorings $k$ of $G$, but the $x_i$'s are noncommuting variables. Allowing the variables commute transforms $Y_G$ to $X_G$.

\begin{example} \label{klika}
    Let $K_n$ denote the clique on $n$ vertices. Then, $Y_{K_n}=e_{[n]}$.
\end{example}

Unlike the regular chromatic symmetric function, this one satisfies some form of deletion-contraction, which is a great computational benefit. If $e$ is an edge of $G=(V, E)$, we write $G\setminus e$ for $(V, E\setminus \{e\})$ and $G/e$ for a graph obtained from $G$ by contracting $e$ and identifying endpoints of $e$.

In order to express the deletion-contraction property of $Y_G$, we need a distinguished edge. If the vertices of $G$ are labeled $v_1, v_2, \ldots, v_n$, we would like this edge to be exactly $\{v_{n-1}, v_n\}$. To obtain one such labeling, we define an action of the symmetric group $S_n$ on the set of vertices $v_1, v_2, \ldots, v_n$ with $\delta(v_i)=v_{\delta(i)}$ for $\delta\in S_n$. Since for any graph $G$, $\delta\circ Y_G=Y_{\delta(G)}$ \cite[Proposition 3.3]{SG}, this is not an obstacle.

\begin{theorem}\label{dcy}
    \cite[Proposition 3.5]{SG} (Deletion-contraction for $Y_G$) Let $G=(V, E)$ be any graph with vertices labeled $v_1, v_2, \ldots, v_n$ and let $e=\{v_{n-1}, v_n\}$ be an edge in $G$. Then, \[Y_G=Y_{G\setminus e}-Y_{G/e}\uparrow,\] where the vertex obtained by contraction in $G/e$ is labeled $v_{n-1}$.
\end{theorem}

Let $X=(V,E)$ be a digraph. For a coloring of vertices with positive integers $f:V\rightarrow\mathbb{N}$, a $V$-listing $\sigma=(\sigma_1,\ldots,\sigma_n)\in\Sigma_V$ is called $(f,X)$-{\it friendly} if \[f(\sigma_1)\leq f(\sigma_2)\leq\cdots\leq f(\sigma_n) \ \text{and}\]  \[f(\sigma_j)<f(\sigma_{j+1}) \ \mathrm{for \ each} \ j\in[n-1] \ \mathrm{satisfying} \ (\sigma_j,\sigma_{j+1})\in E.\] Denote by $\Sigma_V(f,X)$ the set of all $(f,X)$-friendly $V$-listings and by $\delta_f:\Sigma_V\rightarrow\{0,1\}$ its indicator function.

 The author in \cite{M} defined the noncommutative version $W_X$ of the Redei-Berge function of a digraph $X$ and named it the \textit{Redei-Berge function in noncommuting variables.} For a digraph $X=(V, E)$ with vertices labeled $v_1, v_2, \ldots, v_n$ in fixed order, its Redei-Berge function in noncommuting variables, denoted as $W_X$, is \[W_X=\sum_{f:V\rightarrow\mathbb{N}}\sum_{\sigma\in\Sigma_V}\delta_f(\sigma)x_{f(v_1)}x_{f(v_2)}\cdots x_{f(v_n)}.\] Letting the variables commute transforms $W_X$ to $U_X$ \cite{M}. As in the commutative case, for a poset $P$ we define $W_P$ as $W_{D_P}$.

 \begin{example}\label{diskretan}
    Let $T_n=(V, \emptyset)$ be the discrete digraph on $n$ vertices. For a function $f: V\rightarrow\mathbb{N}$ such that $f[V]$ contains $k$ values $c_1<c_2< \cdots< c_k$, there are $|f^{-1}[\{c_1\}]|!|f^{-1}[\{c_2\}]|!\cdots |f^{-1}[\{c_k\}]|!$ listings that are friendly with $f$. Therefore, \[W_{T_n}=\sum_{\pi\in\Pi_n} \pi! m_{\pi}=h_{[n]}.\]  
\end{example} 

 This function satisfies some form of the deletion-contraction property, as opposed to the commutative Redei-Berge function.  The \textit{deletion of an edge} $e\in E$ from a digraph $X=(V,E)$ is the digraph $X\setminus e=(V,E\setminus\{e\})$. The \textit{contraction} of a digraph $X=(V, E)$ by $e=(u,v)\in E$ is the digraph $X/e=(V', E')$, where $V'=V\setminus\{u,v\}\cup\{e\}$ and $E'$ contains all edges in $E$ with vertices different from $u,v\in V$ and additionally for $w\neq u,v$ we have
\begin{itemize}
\item $(w,e)\in E'$ if and only if $(w,u)\in E$ and
\item $(e,w)\in E'$ if and only if $(v,w)\in E$.
\end{itemize}

In analogy with the case of $Y_G$, we would like the distinguished edge $e$ of digraph whose vertices are labeled $v_1, \ldots, v_n$ to be exactly $(v_{n-1}, v_n)$. This is always possible since  for any digraph $X$, $\delta\circ W_X=W_{\delta(X)}$ \cite[Proposition 3.6]{M}.

\begin{theorem}\label{dcw} \cite[Theorem 3.7]{M}(Deletion-contraction for $W_X$) Let $X=(V, E)$ be any digraph with vertices labeled $v_1, v_2, \ldots, v_n$ and let $e=(v_{n-1}, v_n)$ be an edge in $X$. Then, \begin{equation}\label{delcon}
    W_X=W_{X\setminus e}-W_{X/e}\uparrow,\end{equation}
where the vertex obtained by contraction in $X/e$ is labeled $v_{n-1}$.
\end{theorem}

If $P$ is a poset and $e=(p, q)$ a covering in $P$, deleting $e$ from $P$ produces a poset which we denote by $P\setminus e$. In that case, $D_{P\setminus e}=D_P\setminus e$.  However, if $e$ is not a covering, $P\setminus e$ lacks transitivity. Similarly, if we contract an edge that is not a covering, we lose antisymmetry. In what follows, if $e$ is a covering in $P$, we denote by  $P/e$ a poset whose assigned digraph is $D_P/e$. This allows us to rewrite the result given in Theorem \ref{dcw} as \[ W_P=W_{P\setminus e}-W_{P/e}\uparrow.\]  In order to prove the generalization of Theorem \ref{incomparability}, we need some simple results concerning the incomparability graph. If $e=\{u, v\}$ for $u, v\in V$ is not an edge of $G=(V, E)$, $G\cup e$ denotes $(V, E\cup \{e\})$. By abuse of notation, $(u, v)$ in $P$, $\{u, v\}$ in $\mathrm{inc}(P)$ and the newly acquired vertices obtained by contracting these edges will all be denoted as $e$ in the following lemma.

\begin{lemma}

    If $P$ is a poset and $e=(u, v)$ a covering in $P$, then
    
    a) $\mathrm{inc}(P\setminus e)=\mathrm{inc}(P)\cup e$.
    
    b) $\mathrm{inc}(P/e)=(\mathrm{inc}(P)\cup e)/e$.
\end{lemma}

\begin{proof}
   Part a) is obvious. For part b), note that in $P/e$ it holds $w\leq_{P/e}e$ if and only if $w\leq_{P} u\leq_P v$, and $e\leq_{P/e}w$ if and only if $u\leq_P v\leq_P w$. In other words, $w$ is in $P/e$ incomparable with $e$ if and only if it is incomparable with at least one of $u, v$ in $P$. Therefore, $\{w, e\}$ is an edge in $\mathrm{inc}(P/e)$ if and only if $\{w, u\}$, or $\{w, v\}$ is an edge in $\mathrm{inc}(P)\cup e$. The definition of the contraction of an edge in a graph completes the proof.
\end{proof}

We are now ready to obtain a generalization of Theorem \ref{incomparability} that gives a connection between the noncommuting versions of the chromatic function and the Redei-Berge function. 

\begin{theorem} \label{nekomutirajuce}
    If $P$ is a poset with vertices labeled $v_1, \ldots, v_n$, then \[Y_{\mathrm{inc}(P)}=\omega(W_P).\]
\end{theorem}

\begin{proof}
    We prove this theorem by induction on the number of comparable pairs in $P$. If $P$ is a discrete poset, then $\mathrm{inc}(P)$ is a clique. Examples \ref{klika} and \ref{diskretan} yield $Y_{\mathrm{inc}(P)}=e_{[n]}$ and $W_P=h_{[n]}$. The fact that $e_{[n]}=\omega(h_{[n]})$ proves the base of this induction.

    Further, if $P$ has at least one comparable pair, then it has at least one covering. We relabel $P$ so that $e=(v_{n-1}, v_n)$ is a covering. Then $W_P=W_{P\setminus e}-W_{P/e}\uparrow$. Since $P\setminus e$ and $P/e$ have less comparable pairs than $P$, the induction hypothesis gives $W_{P\setminus e}=\omega(Y_{\mathrm{inc}(P\setminus e)})=\omega(Y_{\mathrm{inc(P)}\cup e})$ and $W_{P/e}=\omega(Y_{\mathrm{inc}(P/e)})=\omega(Y_{(\mathrm{inc}(P)\cup e)/e})$. Therefore, \[W_P=\omega(Y_{\mathrm{inc(P)}\cup e})-\omega(Y_{(\mathrm{inc}(P)\cup e)/e})\uparrow=\omega(Y_{\mathrm{inc(P)}\cup e}+Y_{(\mathrm{inc}(P)\cup e)/e}\uparrow),\] where the last equality is satisfied according to Lemma \ref{indukcijaomega}. However, Theorem \ref{dcy} gives \[Y_{\mathrm{inc(P)}\cup e}+Y_{(\mathrm{inc}(P)\cup e)/e}\uparrow=Y_{(\mathrm{inc}(P)\cup e)/e}=Y_{\mathrm{inc}(P)}.\] Consequently, $W_P=\omega(Y_{\mathrm{inc}(P)})$, which completes the proof.
\end{proof}

Applying the commuting operator $\rho$ to both sides of $Y_{\mathrm{inc}(P)}=\omega(W_P)$ yields another proof of Theorem \ref{incomparability}. In other words, there is a commutative diagram: 
\begin{center}
    \begin{tikzcd}
Y_{\mathrm{inc}(P)} \arrow[rr, "\omega", maps to] \arrow[dd, "\rho", maps to] &  & W_P \arrow[dd, "\rho", maps to] \\
                                                            &  &                               \\
X_{\mathrm{inc}(P)} \arrow[rr, "\omega", maps to]                             &  & U_P                            
\end{tikzcd}
\end{center}

\begin{remark}
     The connection between the noncommuting versions of the Redei-Berge and the chromatic function described previously, does not necessarily imply that the deletion-contraction property from Theorem \ref{dcy} and the one from Theorem \ref{dcw} are related. As we have already noted, if $P$ is a poset, the only edges of $D_P$ whose deletion and contraction allow us to remain in the category of posets are the edges that represent a covering. In other words, deletion-contraction of $W_{D_P}$ and deletion-contraction of $\mathrm{inc}(P)$ are two sides of the same coin only when the chosen edge of $D_P$ is a covering. 
     \end{remark}

  \section{Positivity questions}

One of the goals of algebraic combinatorics is finding combinatorial interpretation of coefficients of some symmetric function when expanded in one of the natural bases of $Sym$. If $u=\{u_i\}$ is a basis of some vector space $V$ and $v\in V$, we say that $v$ is $u-$\textit{positive} if every $[u_i]v$ is nonnegative.

  \subsection{Stanley-Stembridge and unit interval orders}

 Stanley proposed a hypothesis that has been one of the central questions of algebraic combinatorics for the last thirty years.  We say that a poset $P$ is $(a+b)-$\textit{free} if $P$ does not contain an induced subposet isomorphic to a disjoint union of an $a-$element chain and a $b-$element chain. 

\begin{conjecture} \cite[Conjecture 5.1]{SR} \label{hipoteza}
    (Stanley-Stembridge) If $P$ is $(3+1)-$free, then $X_{\mathrm{inc}(P)}$ is $e-$positive.
\end{conjecture}
In order to prove this conjecture, mathematicians have developed a range of combinatorial artillery and new techniques. Since recently, there is a probabilistic proof of it in \cite{H}. In the sequel, we give another possible approach to resolving this conjecture.

\begin{theorem}
    If $P$ is a poset, then $[e_{\lambda}] X_{\mathrm{inc}(P)}=[h_{\lambda}]U_P$. Consequently, $X_{\mathrm{inc}(P)}$ is $e-$positive if and only if $U_P$ is $h-$positive.
\end{theorem}
\begin{proof}
    If $X_{\mathrm{inc}(P)}=\sum c_\lambda e_\lambda$, Theorem \ref{incomparability} yields $U_P=\omega(X_{\mathrm{inc}(P)})=\omega(\sum c_\lambda e_\lambda)=\sum c_\lambda h_{\lambda}.$
\end{proof}

Therefore, Conjecture \ref{hipoteza} is equivalent to the following.

\begin{conjecture}
    If $P$ is $(3+1)-$free, then $U_P$ is $h$-positive.
\end{conjecture}

In \cite{G}, Gasharov proved that if $P$ is a $(3+1)-$free poset, then $X_{\mathrm{inc}(P)}$ is $s-$positive. Since $\omega(s_\lambda)=s_{\lambda'}$, we obtain the following.

 \begin{theorem}
     If $P$ is $(3+1)-$free, then $U_P$ is $s-$positive.
 \end{theorem}

It turns out that, to prove Conjecture \ref{hipoteza}, it suffices to prove the $e$-positivity of $X_{\mathrm{inc}(P)}$ whenever $P$ is both $(3+1)-$ and $(2+2)-$free \cite[Theorem 5.1]{MGP}. These posets are so-called \textit{unit interval orders}. Equivalently, a \textit{unit interval order} is a poset $P$ isomorphic to a finite collection $A_P\subseteq \mathbb{R}$, such that for $x, y\in A_P$, $x<_P y\iff x+1<y$, where $<$ denotes the standard order of the real line $\mathbb{R}$ \cite{SS}. A unit interval order $P$ whose elements are labeled $1, \ldots, n$ increasingly with respect to the standard order of the corresponding $A_P\subseteq \mathbb{R}$ is called a \textit{natural unit interval order}.

We say that a graph $G=([n], E)$ is a \textit{natural unit interval graph} if for all $1\leq i<j<k\leq n$, we have $
\{i, k\}\in E\implies\{i, j\}\in E\land\{j, k\}\in E.$ 
    Now, a \textit{unit interval graph} is a graph $G=(V, E)$ for which there exists a renaming of its vertices with numbers from $[n]$ such that the obtained graph is a natural unit interval graph. Clearly, if $P$ is a natural unit interval order, then $\mathrm{inc}(P)$ is a natural unit interval graph. More precisely, it is not difficult to see that unit interval graphs are exactly the incomparability graphs of unit interval orders.

Let $P$ be a natural unit interval order on $[n]$. Then, $D_P=([n], E)$,  where $E$ satisfies that for all $1\leq i\leq j\leq k\leq l\leq n$, \begin{equation} \label{uslov}
    (j, k)\in E\implies (i, l)\in E.
\end{equation}In other words, if the distance between the points labeled $j$ and $k$ is greater than 1, then the same holds for the distance between the points labeled $i$ and $l$. Note that a natural interval order on $[n]$ is an irreducible poset if and only if there does not exist $i\in[n-1]$ such that $(i, i+1)\in E$, since then $P=P|_{[i]}\cdot P|_{\{i+1, \ldots, n\}}$.

\begin{theorem}
    The number of non isomorphic irreducible unit interval orders on $[n]$ for $n\geq 2$ is the same as the number of non-decreasing functions $f: [n-2]\rightarrow[n-1]$ such that $f(i)\geq i$ for every $i\in [n-1]$. 
\end{theorem}

\begin{proof}
   Let $P$ be one irreducible unit interval order on $[n]$. Without loss of generality, suppose that the labeling corresponds to the listing of points in $\mathbb{R}$ in increasing order. For $i\in [n]$, let $N(i)=\{j\in [n]\mid (i, j)\in E\}$. Clearly, $N(n)=\emptyset$, but also $N(n-1)=\emptyset$ since $P$ is irreducible. According to condition from Equation \ref{uslov}, if $j\in N(i)$ and $k>j$, then $k\in N(i)$. Therefore, $\mathrm{min} N(i)$ uniquely determines $N(i)$. We define a function $g: [n-2]\rightarrow \{3, \ldots, n+1\}$ with
   \[
g(i) =
  \begin{cases}
  \mathrm{min} N(i) & \text{ if } N(i)\neq \emptyset \\
  n+1 & \text{ otherwise.} 
  \end{cases}
  \]
  This function is non-decreasing according to Equation \ref{uslov} and $g(i)\geq i+2$ since $P$ is irreducible. It is straightforward to see that such $g$ uniquely determines $P$. Finally, define $f(i)=g(i)-2.$
\end{proof}

We will now recall some well-known graph statistics and introduce their analogues for posets. The \textit{chromatic number} $\chi(G)$ of a graph $G$ is the smallest number $m$ such that there exists a proper coloring of $G$ that uses $m$ colors. The \textit{clique number} of $G$, denoted as $\omega(G),$ is the size of the largest clique in $G$, while its \textit{independence number} $\alpha(G)$ is the size of the largest discrete subgraph of $G$. For a poset $P$, we define its \textit{incomparability number}, denoted as $i(P)$, as the cardinality of the maximal subset of incomparable elements of $P$. The \textit{chain number} $c(P)$ of $P$ is the number of elements of the longest chain in $P$. Clearly, $i(P)=\omega(\mathrm{inc}(P))$ and $c(P)=\alpha(\mathrm{inc}(P))$. It is straightforward to see that in general $\chi(G)\geq \omega(G)$. However, if $G$ in a unit interval order, then $\chi(G)=\omega(G)$ \cite[Proposition 2.12]{ST}.

\begin{theorem} If $P$ is a poset, then $i(P)=\mathrm{min}\{m\in\mathbb{N}\mid u_P(-m)\neq 0\}$. If $P$ is a unit interval order, then $i(P)\geq\frac{|P|}{c(P)}$
\end{theorem}

\begin{proof}
The first part follows from
    $i(P)=\omega(\mathrm{inc}(P))=\chi(\mathrm{inc}(P))=\mathrm{min}\{m\in\mathbb{N}\mid \chi_{\mathrm{inc}(P)}(m)\neq 0\}=\mathrm{min}\{m\in\mathbb{N}\mid u_p(-m)\neq 0\}$, where the last equality holds according to Corollary \ref{polinomi}. 

The inequality is the translation of a well-known inequality $\chi(G)\geq\frac{|G|}{\alpha(G)}.$ 
\end{proof}

\subsection{Application of $W_X$ to the question of $e-$positivity of $X_G$}

The Redei-Berge function in noncommuting variables may be very utile in resolving Stanley-Stembridge conjecture. In this subsection, we give an illustration of this approach.

Since allowing the variables commute transforms $W_X$ to $U_X$, many properties of $U_X$ are just special cases of adequate properties of $W_X$. If $X$ and $Y$ are digraphs with labeled vertices, then we label the vertices of $X\cdot Y$ by listing the vertices of $X$ first in the same order as in $X$ and then the vertices of $Y$ in the same order as in $Y$.

 \begin{theorem}\cite[Theorem 3.5]{M}For any two labeled digraphs, \[W_{X\cdot Y}=W_X\cdot W_Y.\]    
\end{theorem}

\begin{lemma}\label{proizvodd}
    Let $X=(V, E)$ be an $n-$vertex digraph and let $T_m$ be the discrete digraph on $m$ vertices. If $W_X=\sum c_\pi h_\pi$, then $W_{X\cdot T_m}=\sum c_\pi h_{\pi/n+1, \ldots, n+m}$.
\end{lemma}
\begin{proof}
    $W_{X\cdot T_m}=W_X\cdot W_{T_m}=\sum c_\pi h_\pi\cdot h_{[m]}=\sum c_\pi h_{\pi/n+1, \ldots, n+m}.$
\end{proof}



In what follows, we are going to explore the $e-$positivity of the chromatic symmetric function by exploring the $h-$positivity of $W_X$. Example \ref{diskretan} shows that the noncommutative Redei-Berge function of the discrete digraph is $h-$positive. On the other hand, even for some of the simplest digraphs, $W_X$ is not $h-$positive. In order to overcome this difficulty, we need to group the coefficients in the expansion of $W_X$ in $h$ basis that come from set partitions of the same type. 

Let $\pi\in \Pi_n$ be fixed. By $B_{\pi, i}$ we denote the block of $\pi$ that contains $i$. If $(\alpha_1, \ldots, \alpha_l)\models n+1$, let $P(\alpha)$ denote the set of all parititons $\tau=B_1/\ldots/B_l$ of type $\alpha$ such that $\tau\leq \pi+(n+1)$ and $n+1\in B_1$. In the same manner as in \cite[Lemma 5.1]{SG}, from Equation \ref{indukcijah}, we are able to deduce the next theorem.

\begin{theorem} \label{indukovanje}
    If $h_\pi\uparrow=\sum_{\tau\in \Pi_{n+1}} c_\tau h_\tau$, then $c_{\tau}=0$ unless $\tau\leq\pi+(n+1)$ and for any composition $\alpha,$ we have
    \[
\sum_{\tau\in P(\alpha)} c_\tau =
  \begin{cases}
  -\frac{1}{|B_{\pi, n}|} & \text{ if } P(\alpha)=\{\pi/n+1\}\\
\ \   \frac{1}{|B_{\pi, n}|} & \text{ if } P(\alpha)=\{\pi+(n+1)\}\\
 \ \ \ \ \   0 & \text{ otherwise}.
  \end{cases}
  \]
\end{theorem}

For $i\in [n]$, we define an equivalence relation on $\Pi_n$ called \textit{congruence modulo }$i$ by \[\sigma\equiv_i\tau \iff \lambda(\sigma)=\lambda(\tau)\ \land \ |B_{\sigma, i}|=|B_{\tau, i}|.\] We extend this definition so that \[h_\sigma\equiv_i h_\tau\iff \sigma\equiv_i\tau.\] Let $(\tau)$ and $e_{(\tau)}$ denote the equivalence classes of $\tau$ and $e_\tau$ respectively. Clearly,
\[\sum c_\sigma e_\sigma\equiv_i\sum c_{(\tau)}e_{(\tau)}, \text{ where } c_{(\tau)}=\sum_{\sigma\in (\tau)}c_\sigma.\]
Theorem \ref{indukovanje} immediately yields the following.
\begin{corollary}
    For any $\pi\in \Pi_n$, \[h_{\pi}\uparrow\ \equiv_{n+1}\ \frac{1}{|B_{\pi, n}|}h_{(\pi+(n+1))}-\frac{1}{|B_{\pi, n}|}h_{(\pi/n+1)}.\]
\end{corollary}
We are now able to prove our $h-$positivity result.

\begin{theorem}
 Let $X=(V, E)$ be a digraph with vertices labeled $v_1, \ldots v_n$ such that for all $i\in[n-1]$, $(v_n, v_i)\notin E$. If $W_X$ is $h$-positive, then $W_{(X\cdot T_1)\setminus (v_n, v_{n+1})}$ is $h-$positive as well.
\end{theorem}

\begin{proof}
  We apply deletion-contraction to $X\cdot T_1$ and edge $(v_n, v_{n+1})$: \[W_{X\cdot T_1}=W_{X\cdot T_1\setminus (v_n, v_{n+1})}-W_{X\cdot T_1/ (v_n, v_{n+1})}\uparrow.\]  
  According to Lemma \ref{proizvodd}, if $W_X=\sum c_\pi h_\pi$, then $W_{X\cdot T_1}=\sum c_\pi h_{\pi/n+1}$. On the other hand, the condition on vertex $v_n$ gives that $X\cdot T_1/ (v_n, v_{n+1})\cong X$ as digraphs with labeled vertices and hence $W_{X\cdot T_1/ (v_n, v_{n+1})}=W_X.$  Therefore,  since induction operation respects congruence relation \cite[Lemma 6.2]{SG},\[W_{X\cdot T_1\setminus (v_n, v_{n+1})}=W_{X\cdot T_1}+W_{X}\uparrow\ \equiv_{n+1}\ \sum c_{(\tau)}h_{(\tau/n+1)}+\sum c_{(\tau)}h_{(\tau)}\uparrow\]\[\equiv_{n+1}\ \sum c_{(\tau)}\left (1-\frac{1}{|B_{\tau, n}|}\right)h_{(\tau/n+1)}+\sum \frac{c_{(\tau)}}{|B_{\tau, n}|}h_{(\tau+(n+1))}.\] Since $c_{(\tau)}\geq 0$ and $|B_\tau|\geq 1$ for all $\tau$, this finishes the proof.
\end{proof}

 The condition on the vertex $v_n$ given in the previous theorem is met in a large class of digraphs. Namely, if $X$ is a digraph assigned to some poset $P$, then any maximal element of $P$ can play the role of $v_n$. More generally, any acyclic digraph contains at least one such vertex. Also, note that $X$ from the previous theorem does not need to be a digraph assigned to some poset $P$, hence this result cannot be obtained from the results for the chromatic function in noncommuting variables.

 The \textit{lollipop graph} $L_{m, n}$  is obtained by connecting with an edge one vertex of the complete graph $K_m$ and one leaf of the path $P_n$. 

 \begin{corollary}
    Paths and lollipops have $e-$positive chromatic symmetric function.
 \end{corollary}
 \begin{proof}
     If we start with $T_1$ and apply successively the construction from the previous theorem $n-1$ times, we obtain a digraph of a poset whose incomparability graph is $P_n$, the path on $n$ vertices. Similarly, if we start with $T_m$ instead of $T_1$, the obtained digraph is assigned to a poset whose incomparability graph is exactly $L_{m, n}$. Since these digraphs are $h-$positive, their incomparability graphs are $e-$positive.
 \end{proof}

Finally, we note that this approach can be further exploited by applying the generalization of Theorem \ref{indukovanje} on the same way as in \cite{SG}. Clearly, focusing on an edge that is a covering gives us nothing new in comparison to \cite{SG}. However, the advantage of this approach is the possibility to execute the deletion-contraction on $D_P$ with an edge that is not a covering. In that case, the obtained digraph will not be a digraph assigned to some poset. On the other hand, the ambient of digraphs could be simpler for carrying out computations, compared to the ambient of graphs. As we shall see in the upcoming section, there is a nice way to decompose the Redei-Berge function of a digraph. This decomposition does not have its analogue when dealing with the chromatic symmetric function, which supports our belief that working with $U_X$ can sometimes be easier than working with $X_G$.

\section{Bags of sticks}

In this section, we are going to express the Redei-Berge function of any digraph $X$ in terms of the Redei-Berge functions of some simple subdigraphs of $X$.

\subsection{Linear breakdown}

In \cite{M} and \cite{MS}, the authors gave an interesting decomposition technique for the Redei-Berge function. Here, we give a slight improvement of this result. A \textit{bag of sticks} is a digraph that can be written as a disjoint union of directed paths, where paths with no edges are allowed. If $\lambda=(\lambda_1, \ldots, \lambda_k)\vdash n$, we write $P_\lambda$ for a bag of sticks that is a disjoint union of paths $P_{\lambda_1}, \ldots, P_{\lambda_k}$.

\begin{theorem} \cite[Theorem 3.15]{M}\cite[Corollary 4.3]{MS} \label{razbijanje} 
    Let $X=(V, E)$ be a digraph and let $F\subseteq E$ be a subset such that the subdigraph $(V, F)$ is not a bag of sticks. Then,  \begin{equation}\label{podskupovi}W_X=\sum_{\substack{S\subseteq F \\ S\neq\emptyset}}(-1)^{|S|-1}W_{X\setminus S}.\end{equation}
    Consequently, \[U_X=\sum_{\substack{S\subseteq F \\ S\neq\emptyset}}(-1)^{|S|-1}U_{X\setminus S}.\]
\end{theorem}

 Hence, the Redei-Berge function of any digraph that is not a bag of sticks can be expressed as a linear combination of the Redei-Berge functions of appropriate subdigraphs. The same expansion can be applied again to any digraph appearing on the right side of the Equation \ref{podskupovi} that is not a bag of sticks. If we continue with this procedure, we will be able to express the Redei-Berge function of the original digraph as a linear combination of the Redei-Berge functions of its spanning subdigraphs that are bags of sticks. We call this expression the \textit{linear breakdown} of $X$. 

 \begin{example}
     Let $D_0$, $D_1, D_2$ and $D_3$ be the digraphs with vertices $1, 2, 3, 4$ shown in the diagram below. 
    
    
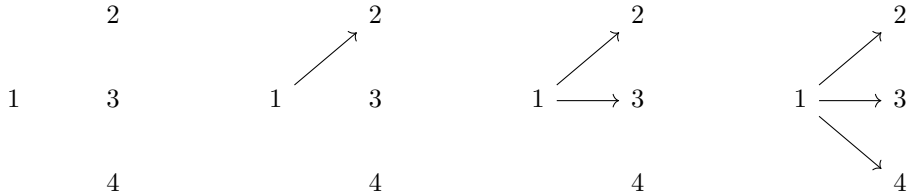
\begin{tikzcd}
  & 2 &  &              & 2 &  &                        & 2 &  &                                   & 2 \\
1 & 3 &  & 1 \arrow[ru] & 3 &  & 1 \arrow[ru] \arrow[r] & 3 &  & 1 \arrow[ru] \arrow[r] \arrow[rd] & 3 \\
  & 4 &  &              & 4 &  &                        & 4 &  &                                   & 4
\end{tikzcd}
 Since $D_3$ is not a bag of sticks, we can apply previous theorem, where we take $F$ to be the whole set of edges of $D_3$. This yields $U_{D_3}=3U_{D_2}-3U_{D_1}+U_{D_0}$. Digraphs $D_1$ and $D_0$ are bags of sticks, so we cannot simplify their Redei-Berge functions. However, $D_2$ is not a bag of sticks and therefore $U_{D_2}=2U_{D_1}-U_{D_0}$. Therefore, $U_{D_3}=3(2U_{D_1}-U_{D_0})-3U_{D_1}+U_{D_0}=3U_{D_1}-2U_{D_0}.$

 \end{example}

  At this point, one question naturally arises : If $X=(V, E)$ is a digraph and $S\subseteq E$ such that $(V, S)$ is a bag of sticks, what is the coefficient of $U_{(V, S)}$ in the linear breakdown of $X$? The answer to this question is given in the upcoming theorem.

 The \textit{edge poset} $\mathcal{E}_X$ of digraph $X=(V, E)$ is a poset on the power set of $E$ defined by the following relation:
 $A<_{\mathcal{E}_X}B\iff A\subsetneq B\land (V, B)$ is not a bag of sticks. In other words, we want to modify relation $\subseteq$ so that the sets of edges that induce bags of sticks are exactly the minimal elements of $\mathcal{E}_X$. If $X=(V, E)$ is a digraph and $S\subseteq E$ such that $(V, S)$ is a bag of sticks $P_\lambda$, we denote $\lambda$ by $\lambda(S)$. 
 
 If $P$ is a poset with a unique minimal element $0_P$ and a unique maximal element $1_P$, let $\mathcal{C}(P)$ denote the set of all chains from $0_P$ to $1_P$. For such poset $P$, we define $\xi(P)$ as
\[
\xi(P)=\sum_{c\in\mathcal{C}(P)}(-1)^{l(c)},
\]where $l(c)$ denotes the length of a chain $c$, i.e. $|c|-1$. Note that for posets $P$ which satisfy the chain condition (all maximal chains from $0_P$ to $1_P$ have the same length), $\xi(P)$ is actually the Möbius function of $P$.

 \begin{theorem} \label{razvoj}
     If $X=(V, E)$ is a digraph that is not a bag of sticks, then 
     \[W_X=\sum_{S\in\mathrm{Min}(\mathcal{E}_X)}(-1)^{|E|-|S|}\xi([S, E])W_{P_{\lambda(S)}}.\]Consequently, \[U_X=\sum_{S\in\mathrm{Min}(\mathcal{E}_X)}(-1)^{|E|-|S|}\xi([S, E])U_{P_{\lambda(S)}}.\]
 \end{theorem}

 \begin{proof}
     If $S\in \mathrm{Min}(\mathcal{E}_X)$, we obtain $U_{(V, S)}$ in the linear breakdown of $X=(V, E)$ once for every chain $S=S_0<_{\mathcal{E}_X}S_1<_{\mathcal{E}_X}\cdots<_{\mathcal{E}_X}S_n=E$ as follows. In the first decomposition of $W_X=W_{(V, E)}=W_{(V, S_n)}$ from Equation \ref{podskupovi}, we focus on $W_{(V, S_{n-1})}$ that appears exactly once. Further, in the decomposition of $W_{(V, S_{n-1})}$, we focus on $(V, S_{n-2})$ that also appears exactly once in it. We continue with this procedure until we get $(V, S_0)=(V, S)$. The coefficient of $U_{(V, S)}$ that arises from this chain is hence $(-1)^{|S_n|-|S_{n-1}|+1}(-1)^{|S_{n-1}|-|S_{n-2}|+1}\cdots (-1)^{|S_1|-|S_0|+1}=(-1)^{|E|-|S|+n}=(-1)^{|E|-|S|}(-1)^{n}$. For the chain $S=S_0<_{\mathcal{E}_X}S_1<_{\mathcal{E}_X}\cdots<_{\mathcal{E}_X}S_n=E$, $n$ is exactly its length. If we consider all possible chains from $S$ to $E$, we obtain $(-1)^{|E|-|S|}\xi([S, E])$, which finishes the proof. \end{proof}

In \cite[Theorem 4.18]{MS}, it is proved that the collection $\{U_{P_\lambda}\}$ forms a linear basis of $Sym$. The previous theorem gives the coefficients in the expansion of $U_X$ in this basis. Note that, if $(V, S)$ is a bag of sticks and $\lambda(S)=\lambda=(\lambda_1, \ldots, \lambda_k)$, then $|S|=\lambda_1-1+\cdots +\lambda_k-1=|V|-l(\lambda)$.

\begin{corollary}
    For a digraph $X=(V, E)$ and for any $\lambda\vdash|V|$, \[
    [U_{P_\lambda}]U_X=(-1)^{|E|+l(\lambda)-|V|}\sum_{\substack{S\in \mathrm{Min}(\mathcal{E}_X)\\ \lambda(S)=\lambda}}\xi([S, E]).
    \]
\end{corollary}

Applying $\omega$ to the result given in Theorem \ref{razvoj} yields the following.

\begin{corollary} \label{hromatskamrazvoj}
If $P$ is a poset and $D_P=(P, E)$ the digraph assigned to $P$, \[Y_\mathrm{inc(P)}=\sum_{S\in\mathrm{Min}(\mathcal{E}_{D_P})}(-1)^{|E|-|S|}\xi([S, E])\omega(W_{P_{\lambda(S)}})\] and \[X_\mathrm{inc(P)}=\sum_{S\in\mathrm{Min}(\mathcal{E}_{D_P})}(-1)^{|E|-|S|}\xi([S, E])\omega(U_{P_{\lambda(S)}}).\]
 \end{corollary}

Note that if $(P, S)$ is a bag of sticks, it is not necessarily a digraph assigned to some poset. More precisely, a bag of sticks is the digraph assigned to some poset $P$ if and only if it consists of disjoint edges and vertices. Therefore, the terms on the right side of the previous corollary cannot always be seen as the chromatic functions of incomparability graphs.

The expansion given in Theorem \ref{razvoj} shows that, in order to resolve the Stanley-Stembridge conjecture, it might be helpful to focus on the structure of $\mathcal{E}_P$, where $P$ is a unit interval order. Since the basis consisting of complete functions is multiplicative, it suffices to prove the $h-$positivity of $U_P$ whenever $P$ is an irreducible unit interval order. 

Finally, we give one more consequence of the decomposition from Theorem \ref{razbijanje} that could be seen as a generalization of the well-known Triple deletion theorem from \cite{OS} in the category of posets. 

\begin{theorem}
    \cite[Theorem 3.1]{OS} (Triple deletion) If $G=(V, E)$ is a graph such that the vertices $u_1, u_2, v$ form a triangle in $G$ and $e_1=\{u_1, v\}$, $e_2=\{u_2, v\}$ then \[X_G=X_{G\setminus e_1}+X_{G\setminus e_2}-X_{G\setminus e_1, e_2}.\]
\end{theorem}

By abuse of notation, we identify $(u, v)$ and $\{u, v\}$ in the following proof.

\begin{theorem}
    Let $P$ be a poset in which $v\in P$ covers some elements $u_1, \ldots, u_k\in P$, $k\geq 2.$ If $F=\{\{u_i, v\}\mid i\in [k]\}$, then
    \[Y_{\mathrm{inc}(P)}=\sum_{S\subseteq F}(-1)^{|S|-1}Y_{\mathrm{inc}(P)\cup S}.\] Consequently, \[X_{\mathrm{inc}(P)}=\sum_{S\subseteq F}(-1)^{|S|-1}X_{\mathrm{inc}(P)\cup S}.\] The same expansions hold if $v\in P$ is covered by $u_1, \ldots, u_k\in P.$
\end{theorem}

\begin{proof}
    Since $F=\{(u_i, v)\mid i\in [k]\}$ is such that the subdigraph $(P, F)$ of $D_P$ is not a bag of sticks, Theorem \ref{razbijanje} yields \begin{equation} \label{pomoc}
        W_{D_P}=\sum_{\substack{S\subseteq F \\ S\neq\emptyset}}(-1)^{|S|-1}W_{D_P\setminus S} \end{equation} 
        
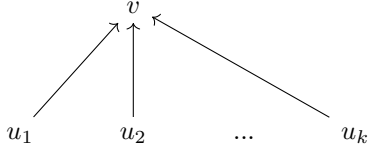
        \begin{center}
            \begin{tikzcd}
                & v              &     &                  \\
                &                &     &                  \\
u_1 \arrow[ruu] & u_2 \arrow[uu] & ... & u_k \arrow[lluu]
\end{tikzcd}
        \end{center}However, $D_P\setminus S= D_{P\setminus S}$, while $\mathrm{inc}(P\setminus S)=\mathrm{inc}(P)\cup S$ and therefore, $\omega(W_{D_P\setminus S})=\omega(W_{D_{P\setminus S}})=Y_{\mathrm{inc}(P\setminus S)}=Y_{\mathrm{inc}(P)\cup S}$. Applying $\omega$ to both sides of Equation \ref{pomoc} completes the proof.
\end{proof}

Note that elements $u_1, \ldots, u_k$ from the previous theorem are necessarily incomparable.  In addition, for $k=2$, this theorem is exactly the Triple deletion theorem  applied to the triangle consisting of vertices $u_1, u_2, v$ in $\mathrm{inc}(P\setminus (u_1, v), (u_2, v))=\mathrm{inc}(P)\cup \{u_1, v\}\cup\{u_2, v\}$.

\subsection{Connection with $\Xi_D$ and $C_D$}

The Redei-Berge function of a bag of sticks can be easily calculated using the results concearning the path-cycle symmetric function $\Xi$ from \cite{TC}. 

If $X=(V, E)$ is a digraph, a subset $S\subseteq E$ is a \textit{path-cycle cover} of $X$ if $(V, S)$ is a disjoint union of directed paths and directed cycles. For a path-cycle cover $S$, we write  $\pi(S)$ for the integer partition whose entries correspond to the sizes of directed paths of $(V, S)$ and $\sigma(S)$ for the integer partition whose entries correspond to the sizes of directed cycles of $(V, S)$.  The \textit{path-cycle symmetric function} of a digraph $X=(V, E)$ is a symmetric function in two countably infinite lists of commuting variables $x=(x_1, x_2, \ldots)$ and $y=(y_1, y_2, \ldots)$ defined as \[\Xi_X(x, y)=\sum_S \widetilde{m}_{\pi(S)}(x)p_{\sigma(S)}(y),\] where the sum is over all path-cycle covers of $X$ and $\widetilde{m}_\lambda=r_1!r_2!\cdots m_\lambda$ for $\lambda=(1^{r_1}2^{r_2} \ldots)$.

In \cite[Proposition 12]{TC} it is proved that for a digraph $X=(V, E)$ \[\Xi_X(x, 0)=\sum_{\sigma\in\Sigma_V}F_{X\mathrm{Des
}(\sigma)^c},\] which together with the fact that $\omega(F_I)=F_{I^c}$ gives the following relation between the path-cycle symmetric function and the Redei-Berge symmetric function.

\begin{theorem}\label{omegovanje}
    For any digraph $X$, \[\Xi_X(x,0)=\omega(U_X).\]
\end{theorem}

 For a loopless digraph $X$, denote by $n_\mu(X)$ the number of spanning subdigraphs of $X$ that are isomorphic to the bag of sticks $P_\mu$. The definition of the path-cycle symmetric function immediately gives the following.

\begin{theorem}\label{TC}
    For any loopless digraph $X=(V, E)$, \[\Xi_X(x,0)=\sum_{\mu\vdash |V|} n_\mu(X)\widetilde{m}_\mu. \] 
\end{theorem}

The number $n_\mu(X)$ has a nice combinatorial interpretation if $X$ is a bag of sticks $P_\lambda$. For a composition $\alpha=(\alpha_1, \ldots, \alpha_k)$ and for $\tau\in \mathfrak{S}_k$, we write $\tau(\alpha)$ for composition $(\alpha_{\tau(1)}, \ldots, \alpha_{\tau(k)})$. If $\lambda, \mu\vdash n$, let $N(\lambda, \mu)$ be the number of permutations $\tau$ of parts of $\mu$ such that $\tau(\mu)$ is finer than $\lambda$, i.e. summing some adjacent parts of $\tau(\mu)$ yields $\lambda$. If $\mu=(1^{r_1}2^{r_2}\ldots)$, it is not difficult to see that \[n_\mu(P_\lambda)=\frac{N(\lambda, \mu)}{r_1!r_2!\cdots},\]which together with Theorem \ref{omegovanje} yields the following.

\begin{corollary} \label{mrazvoj} For any $\lambda\vdash n$,
    \[\omega(U_{P_\lambda})=\sum_{\mu\vdash n} N(\lambda, \mu)m_\mu.\] 
\end{corollary}

\begin{corollary}
    If $P$ is a poset, $D_P=(P, E)$ and $\mu=(1^{r_1}2^{r_2}\ldots)\vdash |P|$, then \[\frac{1}{r_1!r_2!\cdots}\sum_{S\in\mathrm{Min}(\mathcal{E}_{D_P})}(-1)^{|E|-|S|}\xi([S, E]) N(\lambda(S), \mu)\] is the number of stable partitions in $\mathrm{inc}(P).$
\end{corollary}
\begin{proof}
   The expansion of $X_G$ from \cite[Proposition 2.4]{SR} tells that for every graph $G$,
    \[X_G=\sum a_\mu \widetilde{m}_\mu,\] where $a_\mu$ is the number of stable partitions of type $\mu$ in $G.$ Corollaries \ref{hromatskamrazvoj} and \ref{mrazvoj} yield 
    \[X_\mathrm{inc(P)}=\sum_{S\in\mathrm{Min}(\mathcal{E}_{D_P})}(-1)^{|E|-|S|}\xi([S, E])\sum_{\mu\vdash |P|} N(\lambda(S), \mu)m_\mu.\] Comparing the coefficients of $m_\mu$ in these two expansions completes the proof.
\end{proof}

In the end, as an interesting application of the cover polynomial introduced by Chung and Graham in \cite{CG}, we give a formula for calculating the Redei-Berge polynomial of a bag of sticks.

The \textit{cover polynomial} $C_X(m, n)$ of a digraph $X$ is defined as

\[C_X(m, n)=\sum_{S}m(m-1)\cdots (m-l(\pi(S))+1) n^{l(\sigma(S))},\]
where the sum runs over all path-cycle covers of $X$. 

For a loopless digraph $X$, let $n_i(X)=\sum_{l(\mu)=i} n_\mu (X).$ Note that $n_i(X)$ is actually the number of ways to cover $X$ with exactly $i$ disjoint paths. The definition of $C_X$ immediately yields the following.

\begin{theorem}
     Let $X=(V, E)$ be a loopless digraph. Then \[C_X(m, 0)=\sum_{i=0}^{|V|}n_i(X)m(m-1)\cdots(m-i+1).\] 
\end{theorem}

 The cover polynomial $C_X(m, n)$ is a specialization of the path-cycle symmetric function $\Xi$ \cite[Proposition 1]{TC}, which together with Theorem \ref{omegovanje} implies that for $X=(V, E)$, $C_X(m, 0)=(-1)^{|V|}u_X(-m).$ Consequently, \[u_X(m)=\sum_{i=0}^{|V|}(-1)^{|V|-i}n_i(X)m(m+1)\cdots(m+i-1).\] In particular, if $X$ is a bag of sticks, it is not difficult to calculate $n_i(X)$.

\begin{corollary}
    \[u_{P_\lambda}(m)=\sum_{i= l(\lambda)}^{|\lambda|}(-1)^{|\lambda|-i}\binom{|\lambda|-l(\lambda)}{i-l(\lambda)}m(m+1)\cdots (m+i-1).\]
\end{corollary}




\bibliographystyle{model1a-num-names}
\bibliography{<your-bib-database>}




\section{Declarations and statements}

The authors declare that no funds, grants, or other support were received during the preparation of this manuscript. The authors have no relevant financial or non-financial interests to disclose. All authors contributed to the study conception and design. All authors read and approved the final manuscript. 

\section{Data availability}

No data was used for the research described in the article.




\end{document}